\documentclass[12pt]{amsart}

\usepackage[paperwidth=21cm,paperheight=29.7cm,left=2cm,top=2cm]{geometry}
\usepackage[alphabetic]{amsrefs}
\usepackage[T1]{fontenc}
\usepackage{txfonts}
\usepackage{times}
\usepackage{amsmath, amsthm, amssymb, amscd}
\usepackage{verbatim}
\usepackage{tikz}
\usepackage{cite}
\usepackage{color}
\usepackage{a4wide}
\usepackage{graphicx}
\usepackage{wrapfig}
\usepackage{caption}
\usepackage{subcaption}
\usepackage[hidelinks]{hyperref}

\newtheorem{theorem}{Theorem}[section]

\newtheorem{proposition}[theorem]{Proposition}
\newtheorem{corollary}[theorem]{Corollary}
\newtheorem{lemma}[theorem]{Lemma}

\theoremstyle{definition}

\newtheorem{remark}{Remark}%[theorem]{Remark}

\newtheorem{definition}{Definition}

\newcommand{\Map}{\mathrm{Map}}

\newcommand{\Homeo}{\mathrm{Homeo}}

\newcommand{\Ends}{\mathrm{E}}
\newcommand{\Int}{\mathrm{Int}}
\newcommand{\Ext}{\mathrm{Ext}}

\def\N{\mathbb{N}}

\makeatletter
\@namedef{subjclassname@2020}{%
  \textup{2020} Mathematics Subject Classification}
\makeatother

\begin{document}
%\title[Locally CB big MCGs of surfaces with a unique maximal end are CB generated]{Locally CB big mapping class groups of surfaces with a unique maximal end are CB--generated}
\title[On the large scale geometry of big MCGs of surfaces with a unique maximal end]{On the large scale geometry of big mapping class groups of surfaces with a unique maximal end}

%\footnote{}

\author{Rita Jiménez Rolland and Israel Morales}
\date{\today}

\address{
% Israel Morales \newline
Instituto de Matemáticas, UNAM, Unidad Oaxaca, Oaxaca de Juárez, Oaxaca, C.P. 68000. \newline
Oaxaca,  M\'exico.}
% }
\email{rita@im.unam.mx}
 \email{imorales@im.unam.mx}
%\subjclass[2000]{Primary: 37E05. Secondary: 37D40}
%\keywords{Big mapping class group, large scale geometry, geometric group theory}

\begin{abstract}
Building on the work of K. Mann and K. Rafi \cite{MR2019}, we analyze the large scale geometry of big mapping class groups of surfaces with a unique maximal end. We obtain a complete characterization of those that are globally CB, which does not require the \emph{tameness} condition. We prove that, for surfaces with a unique maximal end, any locally CB big mapping class group is CB generated, and we give an explicit criterion for determining which big mapping class groups are CB generated. Finally, we answer \cite[Problem 6.12]{MR2019} by giving an example of a non-tame surface whose mapping class group is CB generated but is not globally CB.

%Building on the remarkable work of K. Mann and K. Rafi \cite{MR2019}, we analyze the large--scale geometry of big mapping class groups of surfaces with a unique maximal end. As results we obtain a complete characterization of those that are globally CB, remarkably, this does not depend on the tameness condition. We also prove that any locally CB big mapping class group is CB--generated. Combining this with a characterization (alternative to that given by Mann and Rafi) of those locally CB big mapping class groups we have an explicit criterion for determining which big mapping class groups are CB--generated, again it does not depend on the tamennes condition. Finally, we answer Mann and Rafi's question in Problem 6.12 by proving the existence of a non-tame surface whose mapping class group is CB-generated but not globally CB.
\end{abstract}

\maketitle
% \begin{center}
%   {Preliminary version. Please do not share.}  
% \end{center}

%\setcounter{tocdepth}{1}
% \tableofcontents

\section{Introduction}
\noindent Let $\Sigma$ be an infinite-type %\footnote{A topological surface is of infinite-type if its fundamental group is not finitely generated. Otherwise, it is called of finite-type.}  
surface and $\Map(\Sigma)$ be the group of all isotopy classes of orientation-preserving self-homeomorphisms of $\Sigma$. This group is called the \emph{(big) mapping class group} of $\Sigma$. If we equip the homeomorphism group with the compact-open topology then $\Map(\Sigma)$ is a Polish group %\footnote{A topological group is \emph{Polish} if it is separable a completely metrizable.} 
with respect to the quotient topology. For an overview about big mapping class groups the reader may consult \cite{AV-Overview}.   

Recently, C. Rosendal \cite{Rosendal2021} %{\color{red}introduced}  
used the notion of \emph{coarsely bounded} sets in order to  extend the framework of geometric group theory to the broader context of topological groups.

\begin{definition}
    Let $G$ be a topological group. A subset of $G$ is \emph{CB (coarsely bounded)} if it has finite diameter for every compatible left-invariant metric on $G$. The group $G$ is \emph{locally CB} if it admits a CB neighborhood of the identity. If $G$ is generated by a CB subset we say that $G$ is \emph{CB generated}.
\end{definition}

With this language, Rosendal extends Gromov's fundamental theorem of geometric group theory showing that if $A_1$ and $A_2$ are two CB generating sets of a Polish group $G$ then the respective word metrics on $G$ are quasi-isometric \cite[Proposition 2.72]{Rosendal2021}. In other words, a CB generated Polish group has a well-defined quasi-isometric type. Moreover, it was proved %\emph{[Ibid.,]} 
that being locally CB is a necessary condition for a Polish group to be CB generated \cite[Theorem 1.2]{Rosendal2021}. 

Based on  Rosendal's framework, K. Mann and K. Rafi started in \cite{MR2019} the study of the large scale geometry of big mapping class groups. They obtained a classification of  locally CB big mapping class groups \cite[Theorem 1.4]{MR2019} and, under the assumption of \emph{tameness}, the classification of  CB generated big mapping class groups \cite[Theorem 1.6]{MR2019}.

%In this paper we give a complete characterization of all locally CB mapping class groups of infinite-type surfaces with a unique maximal end. To state our result we recall the notion of \emph{one maximal end} from \cite{MR2019}. A subset $U$ of an infinite-type surface $\Sigma$ is a \emph{neighborhood} of an end $x$ of $\Sigma$ if it is open and its ends space contains $x$. We say that $\Sigma$ has \emph{one maximal end} if there is an unique end $x$ of $\Sigma$ such that any neighborhood of $x$ has a copy of some neighborhood of any other end $y\neq x$. For a more precise definition see \ref{DEF:UniqueMaximalEnd}.

Studying  the mapping class group of an infinite-type surface $\Sigma$ can be complicated since it involves understanding the homeomorphism group of the space of ends of $\Sigma$, which is denoted by $\Ends(\Sigma)$. In order to address this, Mann and Rafi introduced a preorder on the space of ends of a surface and showed that the induced partial order always has maximal elements.  We say that a open subset $U$ of an infinite-type surface $\Sigma$ is a \emph{neighborhood} of an end $x \in \Ends(\Sigma)$ if its end space contains $x$, that is, $x\in \Ends(U)$. A surface $\Sigma$ has \emph{a unique  maximal end} if there is only one end $x$ of $\Sigma$ such that any neighborhood of $x$ has a copy of some neighborhood of any other end $y$ (see Section \ref{SECTION:Preliminaries} for a precise definition).%and Definition \ref{DEF:UniqueMaximalEnd}). 

In this paper we focus our study on the large scale geometry of big mapping class groups of surfaces with a unique maximal end. These surfaces  have mapping class groups that have exhibited a different behavior from the rest of big mapping class groups. For instance, $\Map(\Sigma)$ has a dense conjugacy class if and only if every finite-type subsurface of $\Sigma$ is displaceable and has a unique maximal end, see \cite{HHetAl2022,LanierVlamis2022}. In addition, some examples of big mapping class groups that do not satisfy the \emph{automatic continuity} property are in this class \cite{MalestinTao2021,Mann2020AC}.

We start our study by giving an alternative characterization of locally CB big mapping class groups for surfaces with a unique maximal end.% attributed to Mann--Rafi, see  \cite[Theorem 1.4]{MR2019}.

\begin{theorem}\label{TEO:BigMapsLocallyCB1maxEnd}%[Big Map's locally CB with a unique maximal end]
    Let $\Sigma$ be an infinite-type surface with a unique maximal end $x$. Then $\Map(\Sigma)$ is locally CB if and only if there is a connected finite-type subsurface $K$ of $\Sigma$ with the following properties: 

    \begin{enumerate}
        \item $\Sigma\setminus K=\Sigma_0\sqcup \Sigma_1\sqcup\cdots \sqcup \Sigma_n$, where the closure of  each $\Sigma_i$ in $\Sigma$ is of infinite-type with one boundary component, $g(\Sigma_i)\in \{0,\infty\}$ and  $x\in E(\Sigma_0)$.%, that is, $\Sigma_0$ is a neighborhood of $x$ in $\Sigma$.
        % \item For each $1\leq i\leq n$, there is $f_{i}\in \Homeo^+(\Sigma)$ such that $f_{i}(\Sigma_i)\subseteq \Sigma_0$.
        \item For any subsurface $U\subseteq \Sigma_0$ that is a neighborhood of $x$ there is $f_U\in Homeo^+(\Sigma)$ such that either $f_U(\Sigma_0)\subseteq U$ or $f_U(\Sigma\setminus U)\subseteq U$.
    \end{enumerate}

    Moreover, if $\Map(\Sigma)$ is locally CB then $\mathcal{V}_K:=\{f\in \Map(\Sigma):\, f\vert_K=Id_K\}$ is a CB neighborhood of the identity.
\end{theorem}

%Our statement of the characterization of the locally CB big mapping class groups with a unique maximal end differs from the statement of Mann--Rafi \cite[Theorem 1.4]{MR2019} for general big mapping class groups. This situation is not isolated, since the class of big mapping class groups of surfaces with a unique maximal end has (in general) a different behaviour from the rest of big mapping class groups. For instance, $\Map(\Sigma)$ has a dense conjugacy class if and only if $\Sigma$ has no non-displaceable finite-type subsurfaces and has an unique maximal end, see \cite{HHetAl2022,LanierVlamis2022}. In addition, the only examples of big mapping class groups that do not satisfy the \emph{Automatic Continuity} property are in this class \cite{MalestinTao2021,Mann2020AC}.

\begin{remark}
Notice that  our condition (1) in Theorem \ref{TEO:BigMapsLocallyCB1maxEnd} is the same as item (1) in \cite[Theorem 1.4]{MR2019}. The differences between the statements lie in item (2). Briefly, for surfaces with a unique maximal end, item (2) in \cite[Theorem 1.4]{MR2019} consists of the following: \emph{i)} $\Ends(\Sigma_0)$ is self-similar, \emph{ii)} each complementary component of $K$ different from $\Sigma_0$ is mapped inside $\Sigma_0$ by a homeomorphism and, \emph{iii)} if $K$ is not empty, for each neighborhood $U\subseteq \Sigma_0$ of the unique maximal end there is a homeomorphism $f$ such that $f(\Sigma_0)\subset U$.  In our result, we do not require conditions \emph{i)} and \emph{ii)}.
 %the self-similarity of the ends space of $\Sigma_0$ (condition \emph{i)}) and condition \emph{ii)} is not needed. 
%  In contrast, in condition \emph{iii)} we don't need any condition on $K$, instead we include the possibility for a neighborhood $U\subseteq \Sigma_0$ of the unique maximal end that there is a homeomorphism $f$ such that $f(\Sigma\setminus U)\subseteq U$. 
%Moreover,  we don't need any condition on $K$ in condition \emph{iii)}, but
Instead, we consider condition \emph{iii)} and include the possibility that for some neighborhoods $U\subseteq \Sigma_0$ of the unique maximal end  there exists a homeomorphism $f_U$ such that $f_U(\Sigma\setminus U)\subseteq U$. This last consideration does not appear in the statement of item (2) of \cite[Theorem 1.4]{MR2019}. We prove %see Theorem \ref{TEO::ClassificationLocCBMAPsWithOneMaximalEnd} below) 
that this possibility can only appear if $\Map(\Sigma)$ is globally CB and the subsurface $K$ of $\Sigma$ can be taken to be empty; see Lemma \ref{LEMMA:SufficientToBeGloballyCB}. 
\end{remark}

% The following corollary is Corollary 5.8 in \cite{MR2019} and it is obtained applying Theorem \ref{TEO:BigMapsLocallyCB1maxEnd}.

% \begin{corollary}\label{CORO:OneMaximalEndCountableEndsSpaceFiniteGenusIsNotLocCB}
%     Let $\Sigma$ be an infinite-type surface with finite and non-zero genus, a unique maximal end $x$ and space of ends countable. Then $\Map(\Sigma)$ is not locally CB.
% \end{corollary}

%As far as we can see, the proof of the Theorem 1.4 in \cite{MR2019} is correct up to the case where the surface has exactly one maximal end. 

We use Theorem \ref{TEO:BigMapsLocallyCB1maxEnd} to give an alternative proof of the self-similarity (see Definition \ref{DEF:SelfsimilarEndsSpace}) of the space of ends of an infinite-type surface with a unique maximal end whose mapping class group is locally CB; see  also \cite[Proposition 5.4]{MR2019}.

\begin{theorem}\label{TEO:OneMaximalEndAndLocallyCBimpliesSpaceEndsSelfSimilar}
    Let $\Sigma$ be an infinite-type surface with a unique maximal end and suppose $\Map(\Sigma)$ is locally CB. Then, the space of ends $\Ends(\Sigma)$ is self-similar.
\end{theorem}

The following notions where introduced in \cite{MR2019} to have some control on the topology of the space of ends  $E(\Sigma)$.

\begin{definition}\label{DEF:StableNeigbTameness}
A neighborhood $U$ of an end $x$ is \emph{stable} \cite[Definition 4.14]{MR2019} if for every neighborhood $U^\prime \subseteq U$ of $x$ there is a homeomorphic copy of $U$ inside $U^\prime$.
\end{definition}

\begin{definition}\label{DEF:TameSurface}
    We say that a surface $\Sigma$ is \emph{tame} if every end of $\Sigma$ which is either of maximal type or any immediate predecessor of an end of maximal type has a stable neighborhood.
\end{definition}

\begin{remark}
This notion was originally introduced in this form by Mann and Rafi, and it has been used in the literature; see for instance \cite[Section 5]{OQR2022}, \cite[Section 5.1]{FGM21}, \cite[Definition 2.15]{SC22}. However, it differs from how it is stated in \cite[Definition 6.11]{MR2019}.  In Section \ref{LastSe} we show that these two definitions are equivalent, when the surface $\Sigma$ has a unique maximal end and $\Map(\Sigma)$ is locally CB  but not globally CB.
\end{remark} 

% {\color{blue}  Our definition of a tame surface is different from how it appears in \cite[Definition 6.11]{MR2019}.in We devote the first part of Section \ref{LastSec} to showing that these two definitions are indeed equivalent {\color{red} when the surface $\Sigma$ has a unique maximal end}.\\}

In \cite{MR2019}, Mann and Rafi give a characterization of globally CB big mapping class groups under the hypothesis of tameness \cite[Theorem 1.7]{MR2019}. We prove that in the case of surfaces with a unique maximal end, the tameness hypothesis is not needed.

\begin{theorem}\label{TEO::ClassificationLocCBMAPsWithOneMaximalEnd}
    Let $\Sigma$ be an infinite-type surface with a unique maximal end and suppose that $\Map(\Sigma)$ is locally CB. Then $\Map(\Sigma)$ is globally CB if and only if the genus of $\Sigma$ is zero or infinite.
\end{theorem}

Thanks to Theorem \ref{TEO::ClassificationLocCBMAPsWithOneMaximalEnd} we have a better understanding of those surfaces with locally but not globally CB mapping class group. In Section \ref{ProofCor} below we prove the following result;  see also Proposition \ref{PROP:ConditionToBeNotLocallyCB} for a more refined statement about the space of ends.
\begin{corollary}\label{CORO:EndsUncountable}
    Let $\Sigma$ be an infinite-type surface with a unique maximal end and suppose that $\Map(\Sigma)$ is locally CB but not globally CB. Then $\Sigma$ has finite nonzero genus and the space of ends of $\Sigma$ is uncountable. 
\end{corollary}

%\begin{corollary}\label{CORO:EndsUncountable}
 %   Let $\Sigma$ be an infinite-type surface with a unique maximal end and suppose that $\Map(\Sigma)$ is locally CB but not globally CB. Then $\Sigma$ has finite nonzero genus and the {\color{blue} set of immediate predecessors of $x$ is not empty. Moreover, the equivalence class of each immediate predecessor of $x$ is uncountable.} space of ends of $\Sigma$ is uncountable. {\color{blue} considerar completar este corolario y la prueba correspondería a la dada en la Proposition \ref{PROP:ConditionToBeNotLocallyCB}}
%\end{corollary}

% {\color{red}  Creo que lo siguiente es la Proposition \ref{PROP:ConditionToBeNotLocallyCB} y deberíamos quitarlo de la Introducción. Si estás de acuerdo, bórralo por favor.

%  Let $\Sigma$ be an infinite-type surface with a unique maximal end $x$ and suppose $\Map(\Sigma)$ is locally CB but not globally CB. Then the set of immediate predecessors of $x$ is not empty. Moreover, the equivalence class of each immediate predecessor of $x$ is uncountable.  }

Under the assumption of tameness,\cite[Theorem 1.6]{MR2019} states that a locally CB big mapping class group is CB generated if and only if the space of ends of the surface is not of \emph{limit type} and is of \emph{finite rank}, see \cite[Definitions 6.2 \& 6.5]{MR2019}. For surfaces $\Sigma$ with a unique maximal end, it can be shown that $\Map(\Sigma)$ is always of finite rank and not of limit type. We see that the tameness hypothesis is actually not needed in order to be CB generated.%it is {\color{red} not the case} {\color{blue} always fulfilled}. 

%Recall that from Mann and Rafi work, under the hypothesis of tameness, a big mapping class group is CB generated if and only if the space of ends of the surface is of \emph{finite rank} and not of \emph{limit type}, \cite[Theorem 1.6]{MR2019}. We have that if a surface $\Sigma$ has a unique maximal end then $\Map(\Sigma)$ is of finite rank and not of limit type. So accordingly to Mann-Rafi classification the only possible obstruction to be CB-generated is the tameness condition. We see that it is not the case. 

\begin{theorem}\label{THM:CBGenerated}
    Let $\Sigma$ be an infinite-type surface with a unique maximal end and suppose that $\Map(\Sigma)$ is locally CB. Then $\Map(\Sigma)$ is CB generated.
\end{theorem}

Recall that a necessary condition for a group to be CB generated is that the group is locally CB. Our Theorem \ref{THM:CBGenerated} shows that it is also a sufficient condition for big mapping class groups of surfaces with a unique maximal end. Recently, T. Hill \cite{Hill2023} observed the same phenomena for pure mapping class groups. 

Additionally, it follows that our Theorem \ref{TEO:BigMapsLocallyCB1maxEnd} and \cite[Theorem 1.4]{MR2019} give  explicit criteria for determining which big mapping class groups are CB generated.
%the reciprocal of the last statement is true in the class of big mapping class groups with a unique maximal end.

\begin{remark}
 It follows from Theorem \ref{THM:CBGenerated} and the work of Horbez, Qing and Rafi  \cite[Theorem 1 and Corollary 2]{OQR2022} that for CB generated big mapping class groups of surfaces $\Sigma$ with a unique maximal end, either 
    \begin{enumerate}
        \item $\Map(\Sigma)$ is globally CB and therefore the quasi-isometric type of $\Map(\Sigma)$ is trivial or,
        \item $\Map(\Sigma)$ admits a continuous and non-elementary action by isometries on a hyperbolic space; in this situation, the space of non-trivial quasi-morphisms of $\Map(\Sigma)$ has infinite dimension.
    \end{enumerate}
\end{remark}

\noindent {\bf Answering a question of Mann and Rafi}. % In  \cite[Problem 6.12]{MR2019}, Mann and Rafi ask about the existence of a non-tame surface whose mapping class group is CB-generated but not globally CB. 
Thanks to our Theorem \ref{THM:CBGenerated} we get a positive answer to \cite[Problem 6.12]{MR2019}.

\begin{theorem}\label{THM:Example}
    There exists a non-tame infinite-type surface whose mapping class group is CB generated but it is not globally CB.  
\end{theorem}
\begin{remark}
In \cite[Example 6.13]{MR2019} Mann and Rafi constructed a non-tame surface with a unique maximal end whose mapping class group is globally CB. We note that in this example the set of immediate predecessors of the unique maximal end is countable. We can modify the Mann-Rafi example by adding finite nonzero genus to obtain a non-tame surface with a unique maximal end. However, by Proposition \ref{PROP:ConditionToBeNotLocallyCB} below, this new surface is not locally CB because the unique maximal end has countably many immediate predecessors. Our example (Section 7) satisfies that the equivalence class of each immediate predecessor of the unique maximal end is uncountable, the unique maximal end has stable neighborhoods, and some of the immediate predecessors of the unique maximal end do not have stable neighborhoods.   
\end{remark}

\medskip

\noindent {\bf Outline}. In Section 2 we present the preliminaries, Section 3 is devoted to the proof of Theorem \ref{TEO:BigMapsLocallyCB1maxEnd}, Section 4 to the proof of Theorem \ref{TEO:OneMaximalEndAndLocallyCBimpliesSpaceEndsSelfSimilar} and Theorem \ref{TEO::ClassificationLocCBMAPsWithOneMaximalEnd}, Section 5 to the proof of Corollary \ref{CORO:EndsUncountable}. Finally, in Section 6 we give the proof of Theorem \ref{THM:CBGenerated}, and in Section 7 we prove Theorem \ref{THM:Example}.\\

\noindent {\bf Acknowledgments}. We are grateful for the financial support of CONAHCYT grant CF 2019-217392. The first author acknowledges funding from a DGAPA-UNAM PASPA sabbatical fellowship and the  second author was supported a DGAPA-UNAM postdoctoral fellowship. We thank K. Mann for kindly answering all our questions, and A. Randecker and J. Hernández-Hernández for helpful comments on an earlier draft of this paper. We are grateful to the referee for suggestions and corrections that improved the exposition.

\section{Preliminaries}\label{SECTION:Preliminaries}

\noindent {\bf Topological surfaces}. 
% By a surface we mean a second countable, Hausdorff topological space which is locally homeomorphic to the plane. 
All our surfaces are assumed to be connected, orientable and possibly with non-empty boundary. The boundary of a surface $\Sigma$ is denoted by $\partial \Sigma$ and always supposed to be compact. A surface is of \emph{finite type} if its fundamental group is finitely generated. Otherwise, we say that it is of \emph{infinite type}. Unless otherwise specified, infinite-type surfaces will be assumed to have empty boundary. Finite-type surfaces are classified, up to homeomorphisms, by their genus, number of punctures and number of boundary components. An infinite-type surface $\Sigma$ with empty boundary is classified, up to homeomorphisms, by their genus (which can be infinite) and a pair of nested topological spaces $\Ends_\infty(\Sigma)\subseteq \Ends(\Sigma)$. The space $\Ends(\Sigma)$ is called the {\it space of ends} of $\Sigma$ and it is homeomorphic to a clopen subset of the Cantor space. The space $\Ends_\infty(\Sigma)$ is a closed subspace of the space of ends and it encodes all ends of $\Sigma$ which are accumulated by genus. Moreover, $\hat{\Sigma}:=\Sigma\cup \Ends(\Sigma)$ is compact and it is called the \emph{Freudenthal compactification} of $\Sigma$. We refer the reader to the work of Richards \cite{Richards1963} and the book of Ahlfors and Sario \cite{AhlforsSario2015} for a detailed discussion about classification of surfaces. 

Any homeomorphism $f:\Sigma \rightarrow \Sigma$ has a unique homeomorphism extension $\hat{f}:\hat{\Sigma}\to \hat{\Sigma}$. In particular, the restriction of $\hat{f}$ to $\Ends(\Sigma)$ induces a homeomorphism of the nested pair $(\Ends(\Sigma),\Ends_{\infty}(\Sigma))$ to itself. From \cite{Richards1963} we have that if the nested pair of spaces $(A,B)\subseteq (\Ends(\Sigma),\Ends_\infty(\Sigma))$ is homeomorphic to the nested pair $(A',B')\subseteq (\Ends(\Sigma),\Ends_\infty(\Sigma))$ then there is a homeomorphism $f:\Sigma \to \Sigma$ such that its extension $\hat{f}$ sends the nested pair $(A,B)$ into $(A',B')$. We assume that any homeomorphism between subsets of $\Ends(\Sigma)$ is induced by a homeomorphism of the surface $\Sigma$. Abusing of notation, we will usually refer simply by $f$ to the homeomorphism of the surface $\Sigma$ or to its extension to $\hat{\Sigma}$. %identify $f$ with its extension to the Freudenthal compactification of $\Sigma$.}

A \emph{simple closed curve} in $\Sigma$ is an embedding of the circle into $\Sigma$. A simple closed curve is \emph{essential} if it is not homotopic to a point,  a puncture, or a boundary component. All curves we consider in this paper will be essential, so we refer to them simply as \emph{curves}. We say that a curve $\alpha$ in $\Sigma$ is \emph{separating} if $\Sigma\setminus \alpha$ is disconnected. %Otherwise, the curve  $\alpha$ is called non-separating.   

By a subsurface of $\Sigma$ we mean a subspace $S\subseteq \Sigma$ that is a surface itself (possibly with non-empty boundary). If we do not specify it in the paper, we assume that all boundary curves of a subsurface are separating curves in $\Sigma$. Furthermore, any subsurface of finite type is assumed to have non-empty boundary. 
\medskip

\noindent {\bf Big mapping class groups}. The \emph{mapping class group} of a surface (of finite or infinite type) $\Sigma$, denoted by $\Map(\Sigma)$, is the group of all isotopy classes of orientation-preserving self-homeomorphisms of $\Sigma$;  if $\partial \Sigma\neq \emptyset$, then we require that all homeomorphisms and
isotopies fix $\partial \Sigma$ pointwise. If we equip the homeomorphism group of $\Sigma$ with the compact-open topology then $\Map(\Sigma)$ is a Polish group with respect to the quotient topology. In recent literature, %the class of the 
the mapping class groups of infinite-type surfaces are often called \emph{big mapping class groups}.  For the rest of the paper, $\Sigma$ denotes an infinite-type surface. Moreover, any homeomorphism $f$ of $\Sigma$ to itself is assumed to be orientation-preserving.

%The \emph{mapping class group} of a surface (of finite or infinite type) $\Sigma$, denoted by $\Map(\Sigma)$, is the group of all orientation-preserving self-homeomorhisms of $\Sigma$ up to isotopy. In recent literature, %the class of the 
%the mapping class groups of infinite-type surfaces are called \emph{(big) mapping class groups}. In general, whether the surface $\Sigma$ is of finite or infinite type, if we equip the homeomorphism group of $\Sigma$ with the compact-open topology then $\Map(\Sigma)$ is a Polish group with respect to the quotient topology. For an overview about big mapping class groups the reader may consult \cite{AV-Overview}. For the rest of the paper, $\Sigma$ denotes an infinite-type surface. Moreover,
%any homeomorphism $f$ of $\Sigma$ to itself is assumed to be orientation-preserving.

Given a subsurface $S\subseteq \Sigma$ we denote by $\mathcal{V}_S$ the subgroup of $\Map(\Sigma)$ defined by all the homeomorphisms $f:\Sigma\rightarrow \Sigma$ such that $f\vert_{S}=Id_S$ up to isotopy. By the Alexander's method \cite{FM12}, if $S$ is a finite-type subsurface of $\Sigma$ then $\mathcal{V}_S$ is an open subgroup. Moreover, the collection of open subgroups $\mathcal{V}_S$ where $S$ runs over all finite-type subsurfaces of $\Sigma$ forms a base of neighborhoods of the identity. Therefore, $\Map(\Sigma)$ is first countable and, in particular, %it makes of 
$\Map(\Sigma)$ is a non-Archimedian\footnote{A Polish group is \emph{non-Archimedian} if the identity has a basis of open subgroups, see \cite{BeckerKechris1996}.} group. Throughout the article, we often use the following fact: for any neighborhood $V$ of the identity in $\Map(\Sigma)$ there is a finite-type subsurface $S\subseteq \Sigma$ such that $\mathcal{V}_S\subseteq V$.

\medskip

\noindent {\bf Large scale geometry of Polish groups}. We use the following characterization of coarsely bounded sets  (CB sets). 
\begin{theorem}[Proposition 2.7, \cite{Rosendal2021}]\label{TEO:Equivalence_CB}
  Let $G$ be a Polish group and $A$ be a subset of $G$. The following are equivalent
  \begin{enumerate}
      \item $A$ is CB.
      \item For every neighborhood $V$ of the identity in $G$, there is a finite subset $F\subseteq G$ and some $k\geq 1$ such that $A\subseteq (FV)^k$.
  \end{enumerate}
\end{theorem}

\medskip

\noindent {\bf Partial order on the space of ends}. We recall the partial order on the space of ends $\Ends(\Sigma)$ of a surface $\Sigma$ introduced by Mann and Rafi in \cite{MR2019}.

\begin{definition}
    Let $x,y\in \Ends(\Sigma)$. We define the binary relation on $\Ends(\Sigma)$ where $y\preceq x$ if for any neighborhood $U_x$ of $x$ in $\Ends(\Sigma)$ there is a neighborhood $U_y$ of $y$ and a homeomorphism $f$ of the surface $\Sigma$ such that $f(U_y)\subseteq U_x$.
\end{definition}

We obtain an equivalence relation on $\Ends(\Sigma)$ declaring that two ends $x,y\in \Ends(\Sigma)$ are of the \emph{same type} if $y\preceq x$ and $x\preceq y$.  Equivalently, $x$ and $y$ are of the same type if and only if there exists a homeomorphism $h$ of $\Sigma$ such that $h(x)=y$; see \cite[Theorem 1.2]{MR2021two}. Define $y\prec x$ if $y\preceq x$ but $x$ and $y$ are not of the same type. The relation $\prec$ defines a partial order on the set of equivalence classes of ends.

\begin{proposition}[Proposition 4.7,\cite{MR2019}]
    The partial order $\prec$ has maximal elements. Moreover, the equivalence class of a maximal element is either finite or a Cantor set.    
\end{proposition}

We denote by $E(x)$ the equivalence class of $x\in \Ends(\Sigma)$ and by $\mathcal{M}(\Sigma)$ the set of all maximal ends for $\prec$.

\begin{definition}[Unique maximal end]\label{DEF:UniqueMaximalEnd} If $\vert \mathcal{M}(\Sigma) \vert =1$ we say that $\Sigma$ has a \emph{unique maximal end}.
\end{definition}
% Equivalently, $\mathcal{M}(\Sigma)=\{x\}$ and for each end $y\neq x$ and each neighborhood $U_x$ of $x$ in $\Ends(\Sigma)$ exists a neighborhood $U_y$ of $y$ and a homeomorphism $f$ such that $f(U_y)\subseteq U_x$.

Mann and Rafi also introduced the notion of self-similar space of ends that we recall now along with some of their results that will be needed for our proofs below.

\begin{definition}\label{DEF:SelfsimilarEndsSpace}
We say that the space of ends $(\Ends(\Sigma),\Ends_\infty(\Sigma))$ of an infinite-type surface $\Sigma$ is \emph{self-similar} if for any decomposition of $\Ends(\Sigma)$ into pairwise disjoint clopen sets $$\Ends(\Sigma)=E_1\sqcup E_2\sqcup \ldots \sqcup E_n$$ there exists a clopen set $D$ in some $E_i$ such that $(D, D\cap \Ends_\infty(\Sigma))$ is homeomorphic to $(\Ends(\Sigma),\Ends_\infty(\Sigma))$.    
\end{definition}

\begin{definition}\label{DEF:Non_DisplaceableSubsurfaces}
Let $\Sigma$ be an infinite-type surface. A finite-type subsurface $K$ of $\Sigma$, possible disconnected, is \emph{nondisplaceable} if for each homeomorphism $f$ of $\Sigma$ we have that $f(K)\cap K\neq \emptyset$.
\end{definition}

\begin{theorem}[Theorem 1.9, \cite{MR2019}]\label{TEO:NonDisplaceableImpliesNotGloballyCB}
Let $\Sigma$ be an infinite-type surface. If $\Sigma$ contains a nondisplaceable finite-type subsurface then $\Map(\Sigma)$ is not globally CB.
\end{theorem}

\begin{proposition}[Proposition 3.1, \cite{MR2019}]\label{PROP:SelfsimilarityImpliesCB}
Let $\Sigma$ be an infinite-type surface of infinite or zero genus. If $E(\Sigma)$ is self-similar then $\Map(\Sigma)$ is globally CB.
\end{proposition}

\begin{lemma}[Lemma 4.12, \cite{MR2019}]\label{LEMMA:UniqueMaximalEndSelsimilarity}
    Suppose $\Sigma$ is an infinite-type surface with a unique maximal end and such that it has no nondisplaceable finite-type subsurfaces. Then $E(\Sigma)$ is self-similar.
\end{lemma}

\begin{theorem}\label{TEO:Classification_CB_Maps_OneMaximalEnd}
    Let $\Sigma$ be an infinite-type surface with zero or infinite genus and with one maximal end. Then $\Map(\Sigma)$ is globally CB if and only if $E(\Sigma)$ is self-similar. 
\end{theorem}
 
\begin{proof}
Suppose that $\Map(\Sigma)$ is globally CB. By Theorem \ref{TEO:NonDisplaceableImpliesNotGloballyCB}, $\Sigma$ does not have displaceable finite-type subsurfaces, then by Lemma \ref{LEMMA:UniqueMaximalEndSelsimilarity}, $E(\Sigma)$ is self-similar. The sufficiency part is given by Proposition \ref{PROP:SelfsimilarityImpliesCB}.    
\end{proof}

% We call a subset $Q\subseteq \Sigma$ \emph{neighborhood} (in $\Sigma$) of an end $y\in E(\Sigma)$  if $y$ is an end of $Q$.

\begin{definition}
Suppose $\Sigma$ has a unique maximal end $x$ and $\alpha$ is a separating curve in $\Sigma$. 

\begin{itemize}
    \item The \textit{interior} of $\alpha$ is defined as the only connected component of $\Sigma \setminus \alpha$ that is a neighborhood of the unique maximal end $x$. We denote it by $\Int(\alpha)$. 
    \item The complement of $\Int(\alpha)\cup \alpha$ in $\Sigma$ is called the \textit{exterior} of $\alpha$ and it is denoted by $\Ext(\alpha)$.
\end{itemize}
\end{definition}

Observe that for each homeomorphism $f$ of $\Sigma$ $$\Int(f(\alpha))=f(\Int(\alpha))\hspace{0.3in}\mbox{and}\hspace{0.3in} \Ext(f(\alpha))=f(\Ext(\alpha)).$$

If $\Sigma$ has a unique maximal end $x$, then any neighborhood of $x$ contains a subsurface with one boundary component whose interior is a neighborhood of $x$. This fact is often used throughout the work. Recall that we are assuming that subsurfaces have separating boundary curves.    

\section{Proof of Theorem \ref{TEO:BigMapsLocallyCB1maxEnd}}

\noindent First we prove the necessity condition of Theorem \ref{TEO:BigMapsLocallyCB1maxEnd} and after three preparatory lemmas we give the proof of the sufficiency part. 

For the necessity part we use the following lemma that is a consequence of \cite[Lemma 5.2]{MR2019}.

\begin{lemma}\label{LEMMA:CRITERIONOTCB}
    Let $\Sigma$ be an infinite-type surface and $K$ be a finite-type subsurface of $\Sigma$. If $\mathcal{V}_K$ is CB then every finite-type subsurface $S$ (possibly disconnected) contained in $\Sigma \setminus K$ is $\Map(\Sigma)$-displaceable.      
\end{lemma}  

\medskip

\noindent \textit{Proof of the necessity part of Theorem \ref{TEO:BigMapsLocallyCB1maxEnd}}. We assume that $\Map(\Sigma)$ is locally CB. Let $V$ be a CB neighborhood of the identity in $\Map(\Sigma)$. Take a connected finite-type subsurface $K$ of $\Sigma$ such that $\mathcal{V}_K\subseteq V$. We have that $\mathcal{V}_K$ is CB. By enlarging $K$ (and therefore, shrinking $\mathcal{V}_K$) if it were necessary, we can assume that $K$ satisfies item (1), that is, the closure of each complementary component of $K$ in $\Sigma$ is of infinite-type with one boundary component either with zero or infinite genus. Without loss of generality, we can assume that the unique maximal end $x$ is an end of $\Sigma_0$. 

Now we prove item (2). Let $U$ be the interior of a connected subsurface of $\Sigma_0$ with one boundary component that is a neighborhood of $x$ in $\Sigma$. If $U$ is isotopic $\Sigma_0$ it is enough to find $f_U$ isotopic to $Id_\Sigma$ such that $f_U(U)=\Sigma_0$. Now, suppose that $U\subseteq \Sigma_0$ is not isotopic to $\Sigma_0$. Then there is a pair of pants $P\subset \Sigma_0$ such that $\partial U \sqcup \partial \Sigma_0 \subseteq \partial P$ and $\Sigma\setminus P=(\Sigma\setminus \overline{\Sigma_0})\sqcup U \sqcup W$. As $\mathcal{V}_K$ is CB, by Lemma \ref{LEMMA:CRITERIONOTCB} there is a homeomorphism $f$ such that $f(P)\cap P=\emptyset$. 

We claim that, up to replacing $f$ by its inverse, we can assume that $f(P)\subset U$. Indeed, observe that either $f(P)\subset U$ or $f(P)\subset \Sigma\setminus U$. If $f(P)\subset \Sigma\setminus U$ then $P\subset \Int(f(\partial U))$. Given that $U=\Int(\partial U)$ and $U$ is a neighborhood of the unique maximal end then $\Int(f(\partial U))=f(\Int(\partial U))=f(U)$; then $f^{-1}(P)\subset U$.

Assume $f(P)\subseteq U$ and we set $f_U:=f$. Again, as $U$ is a neighborhood of $x$ and $\partial U$ is a separating curve in $\Sigma$, there are two possibilities for $f_U(\partial U)$: either $\Int(f_U(\partial U))\subseteq U$ or $\Ext(f_U(\partial U))\subseteq U$. Since $f_U(\Sigma_0)$ is a neighborhood of the unique maximal end $x$ and $\Int(\partial U)\subseteq \Int(\partial \Sigma_0)$, if $\Int(f_U(\partial U))\subseteq U$ then necessarily $\Int(f_U(\partial \Sigma_0))\subseteq U$. In this case we obtain that $f(\Sigma_0)\subseteq U$ because $\Sigma_0=\Int(\partial \Sigma_0)$, see Figure \ref{Fig:2posibilities} a). Finally, since $\Ext(\partial U)=\Sigma \setminus (U\cup \partial U)$ and $\Ext(f_U(\partial U))=f_U(\Ext(\partial U))$, if $\Ext(f_U(\partial U))\subseteq U$ then  
$f_U(\Sigma \setminus U)\subseteq U$, see Figure \ref{Fig:2posibilities} b). In conclusion, either $f_U(\Sigma_0)\subseteq U$ or $f_U(\Sigma\setminus U)\subseteq U$.\qed

\begin{figure}[!ht]
\begin{center}
	\includegraphics[width=.7\textwidth]{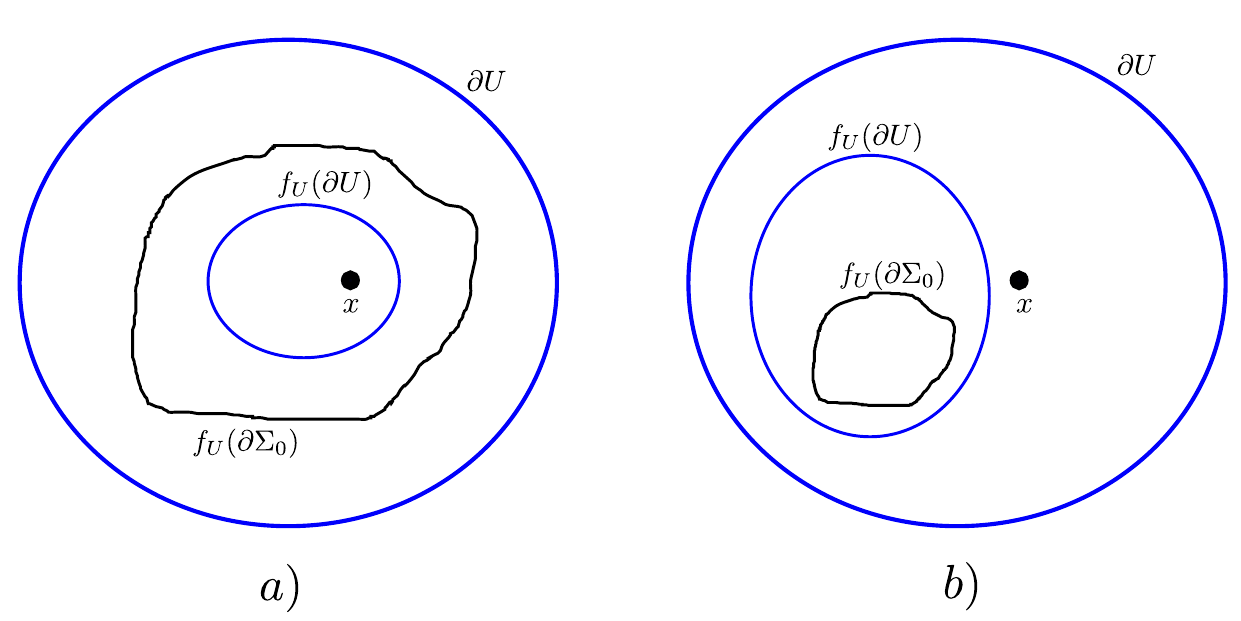}
	\caption{\small a) If $\Int(f_U(\partial U))\subseteq U$ then $f_U(\Sigma_0)\subseteq U$. b) If $\Ext(f_U(\partial U))\subseteq U$ then $f_U(\Sigma \setminus U)\subseteq U$.}
    \label{Fig:2posibilities}
\end{center}
\end{figure}

\medskip

For the proof of the sufficiency part of Theorem \ref{TEO:BigMapsLocallyCB1maxEnd} we use three lemmas. Assume that $\Sigma$ has a unique maximal end $x$ and let $K$ be a connected finite-type subsurface of $\Sigma$ with complementary subsurfaces $\Sigma_0, \Sigma_1,\ldots,\Sigma_n$, i.e, $$\Sigma\setminus K=\Sigma_0\sqcup \Sigma_1\sqcup\cdots \sqcup \Sigma_n.$$

Additionally, suppose that the closure in $\Sigma$ of each $\Sigma_i$ is an infinite-type surface of zero or infinite genus with one boundary component and that $\Sigma_0$ is a neighborhood of the unique maximal end $x$.

In what follows, the {\it support} of a homeomorphism $f:\Sigma \rightarrow \Sigma$, denoted by $\mathrm{supp}(f)$, is defined as the closure in $\Sigma$ of the set $\{s\in \Sigma \mid f(s)\neq s \}$.

\begin{lemma}\label{LEMMA:StableNbhdImpliesCopiesOfComplementsInside}
    Suppose $U\subseteq \Sigma_0$ is a neighborhood of $x$ such that for each subsurface $\widetilde{U}\subseteq U$ that is a neighborhood of $x$ there is a homeomorphism $f_{\widetilde{U}}$ such that $f_{\widetilde{U}}(\Sigma_0)\subseteq \widetilde{U}$. Then for each $1\leq i \leq n$ there exists a homeomorphism $f_i$ such that $f_i(\Sigma_i)\subseteq U$.
\end{lemma}

\begin{remark}
The hypothesis of Lemma \ref{LEMMA:StableNbhdImpliesCopiesOfComplementsInside} implies that $U$ is a stable neighborhood of $x$, see Definition \ref{DEF:StableNeigbTameness}. % \footnote{Given $y\in E(\Sigma)$, a neighborhood $U$ of $y$ in $E(\Sigma)$ is \textit{stable} if for any smaller neighborhood $U'\subseteq U$ of $y$ there is a homeomorphic copy of $U$ contained in $U'$.} 
We point out that Lemma \ref{LEMMA:StableNbhdImpliesCopiesOfComplementsInside} can be obtained using \cite[Lemma 4.18]{MR2019}. Here we provide a self-contained proof.  
\end{remark}

\begin{proof} Let $1\leq i\leq n$ be fixed. Observe that $\Ends(U)$ contains a homeomorphic copy of $\Ends(\Sigma_i)$. Indeed, given that $U$ is a neighborhood of the unique maximal end $x$ and $\Ends(\Sigma_i)$ is a compact subset of $\Ends(\Sigma)$ then there is a finite collection $\{N_j\}_{j=1}^{m}$ of disjoint clopen subsets that covers $\Ends(\Sigma_i)$ and such that each $N_j$ is mapped inside $\Ends(U)$ by a homeomorphism $h_j$. Now, we use the hypothesis of the lemma to make the collection $\{h_j(N_j)\}_{j=1}^{m}$ disjoint inside of $\Ends(U)$.

Now, since $\Sigma_i$ has zero or infinite genus, we can find a homeomorphic copy $\Sigma_i^\prime$ of $\Sigma_i$ contained in $U$. In order to obtain the desired homeomorphism $f_i$, let $P \subseteq \Sigma$ be a pair of pants such that $\partial P$ contains the boundary curves of $\Sigma_i$ and  $\Sigma_i^\prime$, see Figure \ref{Fig:constructhomeo}. Then $f_i$ is the homeomorphism supported on $P\cup \Sigma_i \cup \Sigma_i^\prime$ and sends $\Sigma_i$ onto $\Sigma_i^\prime$. 

\begin{figure}[!ht]
\begin{center}
	\includegraphics[width=.50\textwidth]{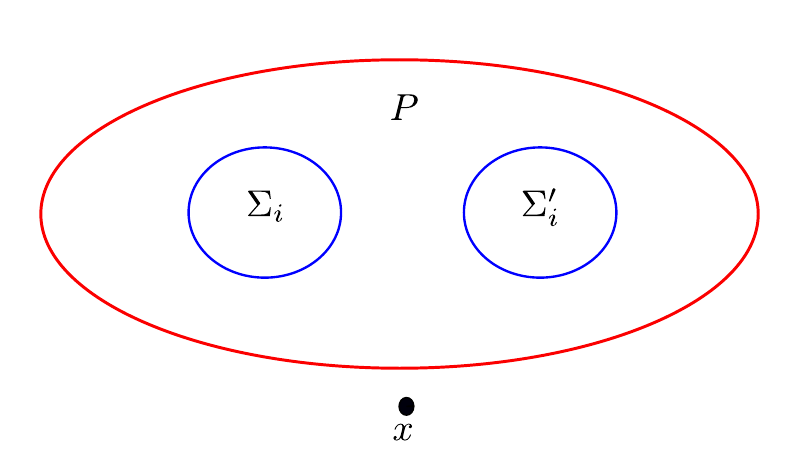}
	\caption{\small $\Sigma_i$ and $\Sigma_i^\prime$ are homeomorphic subsurfaces in $\Sigma$. Then there is a homeomorphism $f$ with support in $P\cup \Sigma_i \cup \Sigma_i^\prime$ that sends $\Sigma_i$ onto $\Sigma_i^\prime$.}
	\label{Fig:constructhomeo}
\end{center}
\end{figure}
  
\end{proof}

% Given that $\Sigma$ has only one maximal end $x$ then every $y\in \Ends(\Sigma_i)$ has a clopen neighborhood $N_y\subseteq E(\Sigma_i)$ that is mapped into $\Ends(W)$ by a homeomorphism. Since $\Ends(\Sigma_i)$ is compact, we can cover it by a finite and pairwise disjoint family of clopen sets $\{N_{y_j}\}_{j
% =1}^l$ such that each $N_{y_j}$ is mapped into $\Ends(W)$ by a homeomorphism $f_{y_j}$\comis{put reference}. Let $U_1\subseteq W$ be a subsurface with one boundary component, neighborhood of $x$ and such that $\Ends(\Sigma_0\setminus U_1)$ contains $f_1(N_{y_1})$. We set $h_1=Id_\Sigma$. By hypothesis there is a homeomorphism $h_2$ with $h_2f_{y_2}(N_{y_2})\subseteq \Ends(U_1)$. Similarly, let $U_2 \subseteq U_i$ be a subsurface with one boundary component, neighborhood of $x$ and such that $\Ends(U_1\setminus U_2)$ contains $h_2f_{y_2}(N_{y_2})$. Again, by hypothesis there is a homeomorphism $h_3$ with $h_3f_{y_3}(N_{y_3})\subseteq \Ends(U_2)$. Inductively, for each $1\leq j\leq l$ we can find a homeomorphism $h_j$ such that $\{h_jf_{y_j}(N_{y_j})\}_{j=1}^l$ is a collection of pairwise disjoint subsets of $\Ends(W)$. The set $\bigcup_{j=1}^l h_jf_{y_j}(N_{y_j})$ is the homeomorphic copy of $E(\Sigma_i)$ contained in $E(W)$.

\begin{lemma}\label{LEMMA:SufficientToBeGloballyCB}
Let $W$ be a subsurface of $\Sigma$ that is a neighborhood of $x$ (may be equal to $\Sigma$) and suppose that for each subsurface $U\subseteq W$ that is a neighborhood of $x$ there is a homeomorphism $f_U$ such that $f_U(\Sigma \setminus U)\subseteq U$. Then $\Map(\Sigma)$ is globally CB.
\end{lemma}

\begin{proof} We prove that any finite-type subsurface of $\Sigma$ is displaceable; in particular, $\Sigma$ has zero or infinite genus. Let $S$ be a finite-type subsurface of $\Sigma$. Then we can construct a connected subsurface $U_S\subseteq W$ that defines a neighborhood of $x$ and such that $U_S\cap S=\emptyset$. Applying the hypothesis of the lemma, there is a homeomorphism $f$ such that $f(\Sigma\setminus U_S)\subseteq U_S$. As $S\subseteq \Sigma \setminus U_S$ we obtain that $f(S)\cap S=\emptyset$.

Now, by Lemma \ref{LEMMA:UniqueMaximalEndSelsimilarity}, $\Ends(\Sigma)$ is self-similar and, by Theorem \ref{TEO:Classification_CB_Maps_OneMaximalEnd}, $\Map(\Sigma)$ is globally CB.
\end{proof}

% \comis{Here an alternative proof. By the discussion above, let $f_T\in \Homeo^{+}(\Sigma)$ with $f_T(T)\subseteq U_T$. Remember that $U_T$ is the connected component of $\Sigma\setminus T$ whose ends space contains $x$. Define $F:=\{f_T^{\pm 1}\}$ and we prove that $g\in (FV_T)^4$. To do this we take into account the following three facts that are easy to verify: 

% \begin{itemize}
%     \item[i)] There is a finite-type subsurface $T^\prime$ of $\Sigma$ such that $T\subseteq T^\prime$ and $g(T)\subseteq T^\prime$.
%     \item[ii)] There is $w\in V_T$ such that $wf_T(T)\subseteq U_{T^\prime}$ where $U_{T^\prime}$ is the connected component of $\Sigma\setminus T^\prime$ whose ends space contains $x$.
%     \item[iii)] there is $u\in V_{wf_T(T)}$ such that $ug\in V_T$.
% \end{itemize}

% Now, we are able to give the conclusion. By iii) $(wf_T)^{-1}u(wf_T)\in V_T$. Hence $u=wf_T\phi_Tf_T^{-1}w^{-1}$ for some $\phi_T\in V_T$. So, by ii), $u\in (FV_T)^3$ and therefore $g\in u^{-1}V_T\subseteq (FV_T)^4$.}
\begin{lemma}\label{LEMMA:V_KisCB}
Let $T$ be a finite-type subsurface of $\Sigma$ containing $K$ and let $U_T$ be the only connected component of $\Sigma\setminus T$ which is a neighborhood of $x$. Suppose that there is a homeomorphism $f_0$ with $f_0(\Sigma_0)\subseteq U_T$, and assume that for each $1\leq i \leq n$ there is a homeomorphism $f_i$ with $f_i(\Sigma_i)\subseteq \Sigma_0$. Then $\mathcal{V}_K\subseteq (F\mathcal{V}_T)^{4n+2}$ where $F=\{f_i^{\pm 1}\}_{i=0}^n$.    
\end{lemma}

\begin{proof}
        Let $g\in \mathcal{V}_K$. Then $g=g_0g_1 \cdots g_n$ where each $g_i$ has its support in $\Sigma_i$. Observe that $f_0^{-1}(T)\subseteq \Sigma \setminus f_0^{-1}(U_T)\subseteq \Sigma\setminus \Sigma_0$. Since $g_0$ has its support on $\Sigma_0$ then $f_0g_0f_0^{-1}\in \mathcal{V}_T$. Therefore $g_0\in (F\mathcal{V}_T)^{2}$. 

    Now, let $1\leq i\leq n$. We notice that $f_ig_if_i^{-1}$ has its support on $\Sigma_0$. This is because $g_i$ has its support on $\Sigma_i$ and $\mathrm{supp}(f_ig_if_i^{-1})=f_i(\mathrm{supp}(g_i))$. As in the previous paragraph we have that $f_ig_if_i^{-1}\in (F\mathcal{V}_T)^2$ and therefore $g_i\in (F\mathcal{V}_T)^4$. Putting all together we obtain the desired result.
\end{proof}

\noindent \textit{Proof of the sufficient part of Theorem \ref{TEO:BigMapsLocallyCB1maxEnd}}. Let $K$ be a finite-type subsurface of $\Sigma$ satisfying the requirements of the statement of Theorem \ref{TEO:BigMapsLocallyCB1maxEnd}. We prove that $\mathcal{V}_K$ is CB. Let $V$ be an arbitrary neighborhood of the identity in $\Map(\Sigma)$. By Theorem \ref{TEO:Equivalence_CB}, we need to show that there exists a finite set $F\subset \Map(\Sigma)$ and $m\geq 1$ such that $\mathcal{V}_K\subseteq (FV)^m$.

Take $T$ a finite-type subsurface of $\Sigma$ where each boundary component is a separating curve, with $K\subseteq T$ and such that $\mathcal{V}_T\subseteq V\cap \mathcal{V}_K$. Define $U_T$ as the unique connected component of $\Sigma\setminus T$ that is a neighborhood of the unique maximal end $x$.

There are two cases for $U_T$:

\begin{itemize}
    \item [\emph{1)}] There is a connected subsurface $U\subseteq U_T$ that is a neighborhood of $x$ for which there exists a homeomorphism $f_U$ such that $f_U(\Sigma_0)\subseteq U$.
    \item[\emph{2)}] If item 1) does not hold then for every $U\subseteq U_T$ a connected subsurface with one boundary component and $x\in E(U)$ there is a homeomorphims $f_U$ with $f_U(\Sigma\setminus U)\subseteq U$.  
\end{itemize}

\noindent Suppose we are in item \emph{1)}. Applying again the hypothesis to $U$ we have two possibilities:

\begin{itemize}
    \item [\emph{i)}]  there is a connected subsurface $\widetilde{U}\subseteq U$ that is a neighborhood of $x$ for which there exists a homeomorphism $f_{\widetilde{U}}$ such that $f_{\widetilde{U}}(\Sigma \setminus \widetilde{U})\subseteq \widetilde{U}$.
    \item[\emph{ii)}] Item \emph{i)} does not hold. Then for every connected subsurface $\widetilde{U}\subseteq U$ that is a neighborhood of $x$ there is a homeomorphism $f_{\widetilde{U}}$ such that $f_{\widetilde{U}}(\Sigma_0)\subseteq \widetilde{U}$.
\end{itemize}

If item \emph{i)} holds, then letting $f_0:=f_U$ and $f_i:=f_{\widetilde{U}}$ for each $1\leq i \leq n$ in Lemma \ref{LEMMA:V_KisCB} we have that $\mathcal{V}_K\subseteq (F\mathcal{V}_T)^{4n+2}\subseteq (FV)^{4n+2}$ where $F:=\{f_i^{\pm 1}\}_{i=0}^{n}$.

Now, suppose item \emph{ii)} holds. We set $f_0:=f_U$ and we use Lemma \ref{LEMMA:StableNbhdImpliesCopiesOfComplementsInside} to obtain for each $1\leq i\leq n$ a homeomorphism $f_i$ such that $f_i(\Sigma_i)\subseteq U\subseteq \Sigma_0$. Applying Lemma \ref{LEMMA:V_KisCB} we conclude that $\mathcal{V}_K\subseteq (F\mathcal{V}_T)^{4n+2}\subseteq (FV)^{4n+2}$ where $F:=\{f_i^{\pm 1}\}_{i=0}^{n}$.

\noindent Finally, if item \emph{2)} holds then by Lemma \ref{LEMMA:SufficientToBeGloballyCB} $\Map(\Sigma)$ is globally CB and, in particular, it is locally CB. \qed

% \begin{remark}
% In the situation \textbf{2)} in the proof of Theorem \ref{TEO:BigMapsLocallyCB1maxEnd} we can obtain a stronger conclusion. In this case, $\Map(\Sigma)$ is globally CB. To see this, as any finite-type subsurface of $\Sigma$ is displaceable, by \cite[Lemma 4.12]{MR2019}, $E(\Sigma)$ is selfsimilar. Finally, by \cite[Proposition 3.1]{MR2019}, $\Map(\Sigma)$ is globally CB.
% \end{remark}

\section{Proof of Theorems \ref{TEO:OneMaximalEndAndLocallyCBimpliesSpaceEndsSelfSimilar} and \ref{TEO::ClassificationLocCBMAPsWithOneMaximalEnd}}

We use the following result.

\begin{lemma}[Lemma 4.10,\cite{MR2019}]\label{LEMMA:AltenativeOfSelfSimilar}
    $\Ends(\Sigma)$ is self-similar if and only if for any decomposition $\Ends(\Sigma)=A_1\sqcup A_2$ into clopen subsets there is some $A_i$ that contains a homeomorphic copy of $\Ends(\Sigma)$.
\end{lemma}

\noindent \textit{Proof of Theorem \ref{TEO:OneMaximalEndAndLocallyCBimpliesSpaceEndsSelfSimilar}}. If $\Sigma$ has no nondisplaceable subsurfaces of finite type then the result is given by Lemma \ref{LEMMA:UniqueMaximalEndSelsimilarity}.

Suppose $\Sigma$ has a nondisplaceable finite-type subsurface $S\subseteq \Sigma$. We use Lemma \ref{LEMMA:AltenativeOfSelfSimilar} to prove that $\Ends(\Sigma)$ is self-similar. Let $\Ends(\Sigma)=A_1\sqcup A_2$ be a decomposition of $\Ends(\Sigma)$ into clopen subsets. Let $K$ be a finite-type subsurface as in Theorem \ref{TEO:BigMapsLocallyCB1maxEnd}. Since $S$ is of finite-type then there is $U\subseteq \Sigma_0$ a neighborhood of the unique maximal end $x$ such that $U\cap S=\emptyset$. Given that $S$ is a nondisplaceable subsurface of $\Sigma$, the subset $U$ satisfies that for any $\widetilde{U}\subseteq U$ that is a neighborhood of $x$ there is a homeomorphism $f_{\widetilde{U}}$ such that $f_{\widetilde{U}}(\Sigma_0)\subseteq \widetilde{U}$. Hence, $U$ is a stable neighborhood of $x$. Without lost of generality we can suppose that $A_1$ contains $x$. If necessary, we can take $U$ small enough such that $\Ends(U) \subseteq A_1$. 
    
By the property satisfied by $U$, there is a homeomorphic copy $\Sigma_0^\prime$ of $\Sigma_0$ contained in $U$. On the other hand, by Lemma \ref{LEMMA:StableNbhdImpliesCopiesOfComplementsInside}, for each $i=1\ldots, n$, there is $\Sigma_i^\prime \subseteq U$ homeomorphic to $\Sigma_i$. Finally, let $P$ denote the set of ends of $\Sigma$ contained into $K$. Remember that $P$ consist of a finite number of punctures. Given that $\Sigma$ has a unique maximal end there is a copy $P^\prime$ of $P$ inside $\Ends(U)$. Now, using again the property of $U$ we can make that the collection $\{ \Sigma_i^\prime\}_{i=0}^n$ pairwise disjoint and disjoint from $P^\prime$. This implies that $\Ends(U)$ (and therefore $A_1$) contains a homeomorphic copy of $\Ends(\Sigma)$.\qed

\medskip

\noindent \textit{Proof of Theorem \ref{TEO::ClassificationLocCBMAPsWithOneMaximalEnd}}. If $\Sigma$ has zero or infinite genus, combine Theorems \ref{TEO:OneMaximalEndAndLocallyCBimpliesSpaceEndsSelfSimilar} and \ref{TEO:Classification_CB_Maps_OneMaximalEnd} to conclude that $\Map(\Sigma)$ is globally CB. If $\Sigma$ has finite non-zero genus then $\Map(\Sigma)$ is not globally CB by Theorem \ref{TEO:NonDisplaceableImpliesNotGloballyCB}.\qed

\section{Proof of Corollary \ref{CORO:EndsUncountable}}\label{ProofCor}
We first prove the following corollary of Theorem \ref{TEO:BigMapsLocallyCB1maxEnd}.
%This proves the following:    

\begin{corollary}\label{COROLLARY:BigMapsLocCBFiniteGenusOneMaximalEnd}
    Let $\Sigma$ be an infinite-type surface with a unique maximal end $x$ and $0<g(\Sigma) < \infty$. Then $\Map(\Sigma)$ is locally CB if and only if there is a connected finite-type subsurface $K$ of $\Sigma$ with the following properties: 

    \begin{enumerate}
        \item $\Sigma\setminus K=\Sigma_0\sqcup \Sigma_1\sqcup\cdots \sqcup \Sigma_n$ where the closure of each $\Sigma_i$ is a surface of infinite-type with one boundary component, $g(\Sigma_i)\in \{0,\infty\}$ and $\Sigma_0$ is a neighborhood of $x$ and,
        % \item For each $1\leq i\leq n$, there is $f_{i}\in \Homeo^+(\Sigma)$ such that $f_{i}(\Sigma_i)\subseteq \Sigma_0$.
        \item for any subsurface $U\subseteq \Sigma_0$ that is a neighborhood of $x$ there is a homeomorphism $f_U$ such that $f_U(\Sigma_0)\subseteq U$.
    \end{enumerate} 
\end{corollary}
\begin{proof}
Suppose that $\Map(\Sigma)$ is locally CB and $K$ is a finite-type subsurface as in Theorem \ref{TEO:BigMapsLocallyCB1maxEnd}. If additionally $0<g(\Sigma) < \infty$, then all the genus of $\Sigma$ is contained in $K$ and therefore for any subsurface $U\subseteq \Sigma_0$ whose interior is a neighborhood of the unique maximal end $x$ does not exist a homeomorphism $f$ satisfying that $f(\Sigma\setminus U)\subseteq U$. \end{proof}

\noindent \emph{Proof of Corollary \ref{CORO:EndsUncountable}}. (By contradiction) Suppose that $\Ends(\Sigma)$ is countable. We prove that $\Map(\Sigma)$ is not locally CB by showing that item (2) in Corollary \ref{COROLLARY:BigMapsLocCBFiniteGenusOneMaximalEnd} does not occur for any finite-type subsurface $K$ satisfying item (1). So, let $K$ be a finite-type subsurface of $\Sigma$ satisfying (1) of Corollary \ref{COROLLARY:BigMapsLocCBFiniteGenusOneMaximalEnd}. As $E(\Sigma)$ is homeomorphic to $\omega^\alpha + 1$ with $\alpha$ a countable ordinal and $\Sigma_0$ is a neighborhood of the unique maximal end $x$. The ordinal $\alpha$ can be either a successor ordinal or a limit ordinal. In either case, $E(\Sigma\setminus \Sigma_0)$ has a finite number of maximal ends. Taking $U\subseteq \Sigma_0$ neighborhood of $x$ such that $E(\Sigma\setminus U)$ contains more maximal ends than $E(\Sigma\setminus \Sigma_0)$ we have that for this $U$ there does not exist a homeomorphism $f_U$ such that $f_U(\Sigma_0)\subseteq U$.  \qed

\section{Proof of Theorem \ref{THM:CBGenerated}}

A globally CB group is in particular CB generated. Suppose that $\Map(\Sigma)$ is locally CB but not globally CB. By Theorem \ref{TEO::ClassificationLocCBMAPsWithOneMaximalEnd}, the surface $\Sigma$ has finite nonzero genus. Let $K\subseteq \Sigma$ be a finite type subsurface as in Theorem \ref{TEO:BigMapsLocallyCB1maxEnd}, that is, $$\Sigma\setminus K=\Sigma_0\sqcup \Sigma_1\sqcup\cdots \sqcup \Sigma_n,$$ where each $\Sigma_i$ is of infinite type and has genus zero, $\Sigma_0$ is a neighborhood of the unique maximal end of $\Sigma$ and, $\mathcal{V}_K$ is a CB neighborhood of the identity. Denote by $P_K$ to the union of all $\Sigma_i$ for $i=1,\ldots, n$. Observe that $E(\Sigma)$ is self-similar (by Theorem \ref{TEO:OneMaximalEndAndLocallyCBimpliesSpaceEndsSelfSimilar}), $K$ is compact with $g(K)=g(\Sigma)$ and it has $n+1$ boundary components. 

% Let $L\subseteq K$ be a compact surface of genus equal to the genus of $K$ and with one boundary component that separates $\Sigma$ such that the closure of $K\setminus L$ is homeomorphic to the sphere with $n+2$ holes. Denote by $\Sigma_L$ the unique complementary component of the interior of $L$ in $\Sigma$.

% \begin{figure}[!ht]
% \begin{center}
% 	\includegraphics[width=.7\textwidth]{CBgeneratedMAP.jpg}
% 	\caption{\small The mapping class group of the surface $\Sigma$ is locally CB and not globally CB. Then, the genus of $\Sigma$ is finite not zero, the space of ends $\Ends(\Sigma)$ is selfsimilar and uncountable and, $\mathcal{V}_K$ is a CB neighborhood of the identity.}
% 	\label{Fig:CBgenerated}
% \end{center}
% \end{figure}

% \begin{proposition}
% [Lemma 6.8, \cite{MR2019}]\label{Lemma:V_LisCBgenerated}
%     $\mathcal{V}_L$ is CB generated.
% \end{proposition}

% The proof of Proposition \ref{Lemma:V_LisCBgenerated} is based on the following two lemmas.

% \begin{lemma}\label{Lemma:HomeoinPMap=IdentityOnFiniteSubsurfaces}
%     Let $f\in \PMap(\Sigma)$, $S$ be a finite-type subsurface of $\Sigma$ and $S^\prime:=S\cup f(S)$. Then there is $g\in \Map(S^\prime)$ such that $gf\vert_S=Id_S$.
% \end{lemma}

% \begin{proof}
%     \comis{Hasta aquí}
% \end{proof}

We use the following lemma that appears in \cite[Observation 6.9]{MR2019}. By abuse of notation, given a finite-type subsurface $S$ of $\Sigma$ with compact boundary we think of an element of $\Map(S)$ as an element of $\Map(\Sigma)$ by extending it by the identity on the complement of $S$ in $\Sigma$.

\begin{lemma}[Observation 6.9, \cite{MR2019}]\label{Lemma:CBGenerationOfClassicalMCG}
    Let $\Sigma$ be an infinite-type surface possibly with nonempty boundary and $S\subseteq \Sigma$ be a finite-type subsurface. Then there is $D_S\subseteq \Map(\Sigma)$ a finite set of Dehn twist such that for any finite-type subsurface $S^\prime \subseteq \Sigma$ $$\Map(S^\prime)\subseteq \langle D_S\cup \mathcal{V}_S \rangle \subseteq \Map(\Sigma).$$  
\end{lemma}

\noindent \emph{Proof of Theorem \ref{THM:CBGenerated}}. Applying Lemma \ref{Lemma:CBGenerationOfClassicalMCG} to the subsurface $K$, let $D_K$ be a finite set of Dehn twists such that for every finite type subsurface $S^\prime$ of $\Sigma$, $\Map(S^\prime)$ is contained in the group generated by $D_K\cup \mathcal{V}_K$. As $\Ends(\Sigma)$ is self-similar, there is $g_K\in \Map(\Sigma)$ such that $g_K(P_K)\subseteq \Sigma_0$. Let $G$ be the subgroup of $\Map(\Sigma)$ generated by the CB set $\{g_K\}\cup D_K \cup \mathcal{V}_K$. We show that $\Map(\Sigma)$ coincides with $G$.

Let $f\in \Map(\Sigma)$. First we prove that there exist $f^\prime, f^{\prime \prime} \in G$ such that $f^\prime f^{-1} f^{\prime \prime}\vert_{P_K}=Id_{P_K}$. Indeed, take $U\subseteq \Sigma_0$ a neighborhood of the unique maximal end $x$ of $\Sigma$ such that $f(U)\subseteq \Sigma_0$. As the space of ends of $\Sigma$ is self-similar, $\Ends(U)$ contains a homeomorphic copy of $\Ends(\Sigma)$, and therefore there is $P_K^\prime \subseteq U$ homeomorphic to $P_K$. In particular, $f(P_K^\prime)\subseteq \Sigma_0$. Now, as $\Sigma_0$ has genus zero there is $h\in \Map(\Sigma_0)$ such that $hf(P_K^\prime)=g_K(P_K)$. So, $f^{-1}h^{-1}g_K (P_K)\subseteq \Sigma_0$. As $g_K(P_K)$ is also contained into $\Sigma_0$, there is $h^\prime \in \Map(\Sigma_0)$ such that $h^\prime f^{-1}h^{-1}g_K (P_K)=g_K(P_K)$ and $h^\prime f^{-1}h^{-1}g_K\vert_{\partial P_K}=g_K\vert_{\partial P_K}$, in other words, $g_K^{-1} h^\prime f^{-1}h^{-1}g_K(P_K)=P_K$ and $g_K^{-1} h^\prime f^{-1}h^{-1}g_K\vert_{\partial P_K}=Id_{\partial P_K}$. Finally, we can find an element $w\in \Map(P_K)\subseteq \mathcal{V}_K$ such that the restriction of $wg_K^{-1} h^\prime f^{-1}h^{-1}g_K$ to $P_K$ is equal to $Id_{P_K}$. Letting $f^\prime:=wg_K^{-1} h^\prime $ and $f^{\prime \prime}:= h^{-1}g_K$ we obtain the desired result.

% \begin{figure}[!ht]
% \begin{center}
% 	\includegraphics[width=.7\textwidth]{HomeosCBgenerado.pdf}
% 	\caption{\small \color{blue} incluir descripción}
% 	\label{Fig:HomeoCBgenerado}
% \end{center}
% \end{figure}

Let $g:=f^\prime f^{-1} f^{\prime \prime}$ and put $S^\prime:=K\cup g(K)$. Again, by Lemma \ref{Lemma:CBGenerationOfClassicalMCG}, $\Map(S^\prime)$ is contained in the group generated by $D_K\cup \mathcal{V}_K$ and then it is contained in the group $G$. Now, observe that $\partial \Sigma_0$ and $g(\partial \Sigma_0)$ are essential separating curves of the same topological type in $S^\prime$. Then there is $g^\prime\in \Map(S^\prime)\subseteq G$ such that $g^\prime g$ is the identity on $K$, that is, $g^\prime g \in \mathcal{V}_K$. Therefore, $f\in G$.\qed

\section{Proof of Theorem \ref{THM:Example}}\label{LastSe}

In the first part of this section we recall the original %Mann and Rafi 
definition \cite[Definition 6.11]{MR2019}  of a tame surface in the case when the surface has a unique maximal end. %adapted to our context. 
 We show that this definition is equivalent to Definition \ref{DEF:TameSurface} given in the Introduction of this paper. %{\color{red} Maybe mention this later:} We note that the proof given here can be adapted to the general context. 
 In the second part of the section we prove  Theorem \ref{THM:Example} by  constructing a non-tame infinite-type surface with CB-generated but not globally CB mapping class group.\\

The following lemma is inspired by the first point of \cite[Lemma 6.10]{MR2019}. 

% In Mann-Rafi paper, additionally to the hypothesis of locally CB, they required that the space of ends to be not of \emph{limit type}.

% We recall that limit type of the space of ends {\color{red} having a space of ends of limit type} is an obstruction to be {\color{red} for being} CB-generated (\cite[Lemma 6.4]{MR2019}). Thus, if $\Map(\Sigma)$ is globally CB then $E(\Sigma)$ is not of limit type.

\begin{lemma}\label{LEMMA:ExistenceOfImmediatePredecessors}
        Let $\Sigma$ be an infinite-type surface with a unique maximal end $x$. Suppose that $\Map(\Sigma)$ is locally CB and that for any $N \subseteq E(\Sigma_0)$ clopen neighborhood of $x$ there is $f_N \in \Homeo^+(\Sigma)$ such that $f_N (E(\Sigma)\setminus N) \subseteq E(\Sigma)\setminus E(\Sigma_0)$ (compare  with Theorem \ref{TEO:BigMapsLocallyCB1maxEnd}, item 2)). Then there exists a clopen neighborhood $N(x)\subseteq E(\Sigma_0)$ of $x$ such that the clopen subset $W_x:=E(\Sigma_0)\setminus N(x)$ satisfies
    \begin{equation}
    E(z)\cap W_x \neq \emptyset \mbox{ if and only if } E(z)\cap (E(\Sigma_0)\setminus \{x\})\neq \emptyset,    
    \end{equation}
    in other words, $W_x$ has representatives of each end in $E(\Sigma_0)\setminus \{x\}$.
\end{lemma}

\begin{proof}
    The proof is done by contradiction. Suppose that for each clopen neighborhood $N\subseteq E(\Sigma_0)$ of $x$ there is an end $z\in E(\Sigma_0)\setminus \{x\}$ such that $E(z)\cap(E(\Sigma_0)\setminus N)=\emptyset$ . Then there exists a decreasing sequence of nested closed neighborhoods of $x$, $N_1  \supseteq N_2 \supseteq \cdots  $, with $\bigcap_{i\in \N} N_i= \{x\}$, $N_i \subseteq E(\Sigma_0)$ and such that for each $i\in \N$ there is $z_i\in E(\Sigma_0) \setminus \{x\}$ with $E(z_i)\cap(E(\Sigma_0)\setminus N_i)=\emptyset$ and $z_i \in N_i \setminus N_{i+1}$. By hypothesis, for each $i\in \N$ there is a homeomorphism $f_i$ such that $f_i(E(\Sigma) \setminus N_{i+1}) \subseteq E(\Sigma) \setminus E(\Sigma_0)$. Consequently, for every $i\in N$ there is $z_i^\prime=f_i(z_i)$ in $E(z_i)$ contained in $E(\Sigma)\setminus E(\Sigma_0)$. Let $y^\prime\in E(\Sigma)\setminus E(\Sigma_0)$ be an accumulation point of the sequence $\{z_i^\prime\}$. Since $x$ is the unique maximal end of $\Sigma$ there exists an end $y$ equivalent to $y^\prime$ contained in $N_1$. Let us take $V$ a neighborhood of $y$ such that $V$ is disjoint from $N_m$ for some $m\in \mathbb{N}$. Recall that $y^\prime$ is equivalent to $y$ and the latter is an accumulation point of the sequence $\{z_i^\prime\}$. Then, for some $j\geq m$, there exists $w\in E(z_j)$ contained in $V$ and  hence  $w\in E(z_j)\cap(E(\Sigma_0)\setminus N_j)$, which is a contradiction. %This contradicts that the complement of $N_j$ in $E(\Sigma_0)$ {\color{red} Quizá se lea mejor si se escribe $E(\Sigma_0)\setminus N_j$ } does not contain representatives of $z_j$.  This contradicts that  $E(z_j)\cap(E(\Sigma_0)\setminus N_j)=\emptyset$
\end{proof}

\begin{remark}
In particular, the hypothesis of Lemma \ref{LEMMA:ExistenceOfImmediatePredecessors} are satisfied if $\Map(\Sigma)$ is locally CB but not globally CB; see Theorem \ref{TEO::ClassificationLocCBMAPsWithOneMaximalEnd} and Corollary \ref{COROLLARY:BigMapsLocCBFiniteGenusOneMaximalEnd}.    Notice that there are cases when $\Map(\Sigma)$ is globally CB  and the argument in Lemma \ref{LEMMA:ExistenceOfImmediatePredecessors} cannot be applied to obtain a clopen $W_x$ neighborhood of $x$. For example, take $\Sigma$ to be the infinite-type surface of genus zero with space of ends homeomorphic to the ordinal number $\omega^\omega + 1$, then $\Map(\Sigma)$ is globally CB. However, it is not possible to find a neighborhood of the unique maximal end of $\Sigma$ such that its complement has representatives of all non maximal ends.  
\end{remark}

With the notation introduced in Lemma \ref{LEMMA:ExistenceOfImmediatePredecessors}, we are able to state \cite[Definition 6.11]{MR2019} for the case of surfaces with a unique maximal end. We refer the reader to Definition \ref{DEF:StableNeigbTameness} for the definition of stable neighborhood. 

\begin{definition}[Tame surface]\label{DEF:TameSurface_MR}
    Let $\Sigma$ be a surface with a unique maximal end $x$ and suppose that $\Map(\Sigma)$ is locally CB and not globally CB. We say that $\Sigma$ is \emph{tame} if the unique maximal end $x$ has a stable neighborhood and every maximal end of $W_x$ has a stable neighborhood.
\end{definition}

Suppose $\Sigma$ has a unique maximal end and $\Map(\Sigma)$ is locally CB  but not globally CB. The equivalence of Definition \ref {DEF:TameSurface_MR} and Definition \ref{DEF:TameSurface} in this case follows from the Proposition \ref{PROP:MaximalEndsW_x=Immediatepredecessor} below.  In the proof of this proposition we repeatedly use the following observation, which is a consequence of the fact that $W_x$ is closed in $E(\Sigma)$. 

\begin{remark}\label{REMARK:W_xClopen}
    If $y,z \in W_x$ then $y \preceq z$ in $W_x$ if and only if $y \preceq z$ in $E(\Sigma)$.
\end{remark}

\begin{proposition}\label{PROP:MaximalEndsW_x=Immediatepredecessor}
    Let $\Sigma$ be an infinite-type surface with a unique maximal end $x$ and suppose that $\Map(\Sigma)$ is locally CB  but not globally CB. Then $y$ is an immediate predecessor of $x$ if and only if there is a representative of $y$ which is maximal in $W_x$.
\end{proposition}

\begin{proof}
    Suppose that $y$ is an immediate predecessor of $x$. By Lemma \ref{LEMMA:ExistenceOfImmediatePredecessors} the clopen $W_x$ has representatives of each end in $E(\Sigma_0)\setminus \{x\}$. Since $x$ is the only accumulation point of $E(y)$,  we can assume, up to equivalence, that $y\in W_x$. We show that $y$ is maximal in $W_x$. Let $z\in W_x$ such that $y\preceq z$ in $W_x$, then by Remark \ref{REMARK:W_xClopen} we have that $y\preceq z$ in $E(\Sigma)$. Given that $y$ is an immediate predecessor of $x$ we conclude that $y$ is equivalent to $z$ in $E(\Sigma)$. Again, by Remark \ref{REMARK:W_xClopen}, we have that $y$ is equivalent to $z$ in $W_x$ and hence $y$ is a maximal end in $W_x$.

    Reciprocally, suppose that $y\in W_x$ is a maximal end of $W_x$. We show that $y$ is an immediate predecessor of $x$. Let $z\in E(\Sigma)\setminus \{x\}$ such that $y \preceq z$ in $E(\Sigma)$. Since $E(y)$ accumulates in $\{x\}$ (because $x$ is the unique maximal end of $\Sigma$) and $W_x$ has representatives of each end in $E(\Sigma_0)\setminus \{x\}$, without loss of generality we can assume that $z$ is contained in $W_x$. By Remark \ref{REMARK:W_xClopen}, it follows that $y \preceq z$ in $W_x$. The maximality of $y$ in $W_x$ implies that $z$ is equivalent to $y$ in $W_x$. Applying again Remark \ref{REMARK:W_xClopen}, we have that $z$ is equivalent to $y$ in $E(\Sigma)$ and therefore $y$ is an immediate predecessor of $x$.  
\end{proof}

%{\color{magenta}Agregar frase de enlace, revisar comentario en Corollary 1.4}

The following result refines the conclusion of Corollary \ref{CORO:EndsUncountable} about the space of ends.  It gave us some insight for constructing the example for the proof of Theorem \ref{THM:Example} below.

\begin{proposition}\label{PROP:ConditionToBeNotLocallyCB}
    Let $\Sigma$ be an infinite-type surface with a unique maximal end $x$ and suppose $\Map(\Sigma)$ is locally CB but not globally CB. Then the set of immediate predecessors of $x$ is not empty. Moreover, the equivalence class of each immediate predecessor of $x$ is uncountable.  
    % Suppose that the set of immediate predecessors of the unique maximal end $x$ is not empty and countable. Then $\Map(\Sigma)$ is not locally CB.
    % If $x$ has countably many (infinite) immediate predecessors then $\Map(\Sigma)$ is not locally CB.   
\end{proposition}

\begin{proof}
The first part is consequence of Lemma \ref{LEMMA:ExistenceOfImmediatePredecessors} and  Proposition \ref{PROP:MaximalEndsW_x=Immediatepredecessor}.  

Now, as $\Map(\Sigma)$ is not globally CB, by Theorem \ref{TEO::ClassificationLocCBMAPsWithOneMaximalEnd}, $\Sigma$ has finite nonzero genus. Let $K$ be the finite-type subsurface of $\Sigma$ satisfying items 1) and 2) of Corollary \ref{COROLLARY:BigMapsLocCBFiniteGenusOneMaximalEnd}. The proof of the moreover part is done by contradiction. Let $y$ be an immediate predecessor of $x$ and suppose that $E(y)$ is countable. Since $x$ is the unique accumulation point of $E(y)$ then $E(\Sigma\setminus \Sigma_0)$ contains a finite number of elements of $E(y)$. Taking a neighborhood $U\subseteq \Sigma_0$ of $x$ such that $E(\Sigma\setminus U)$ contains more elements of $E(y)$ than $E(\Sigma\setminus \Sigma_0)$, we have that for such $U$ there does not exist a homeomorphism $f_U$ such that $f_U(\Sigma_0)\subseteq U$. This contradicts item 2) of Corollary \ref{COROLLARY:BigMapsLocCBFiniteGenusOneMaximalEnd}.      
\end{proof}

% The ideas in the proof of Corollary \ref{CORO:EndsUncountable} can be applied to obtain the following result:

\medskip

\noindent \emph{Proof of Theorem \ref{THM:Example}}. Our example is inspired by the constructions carried out at \cite[Sections 2 \& 3]{MR2021two}. For each $n\in\mathbb{N}$, let $D_n$ be the surface of genus zero with one boundary component and such that: 1) $\Ends(D_n) = C_n\sqcup Q_n$ where $C_n$ is homeomorphic to a Cantor set and $Q_n$ is a countable set, 2) $\Ends(D_n)$ has Cantor-Bendixson rank $n$ and, 3) for each derived set $\Omega$ of $\Ends(D_n)$ that has isolated points, the accumulation set of the isolated points of $\Omega$ contains $C_n$; see Figure \ref{Fig:Pre-Ejemplo}. 
From the construction it follows that the ends in $C_n$ are not comparable with the ends in $C_m$ if $n\neq m$. 

\begin{figure}[!ht]
\begin{center}
	\includegraphics[width=.5\textwidth]{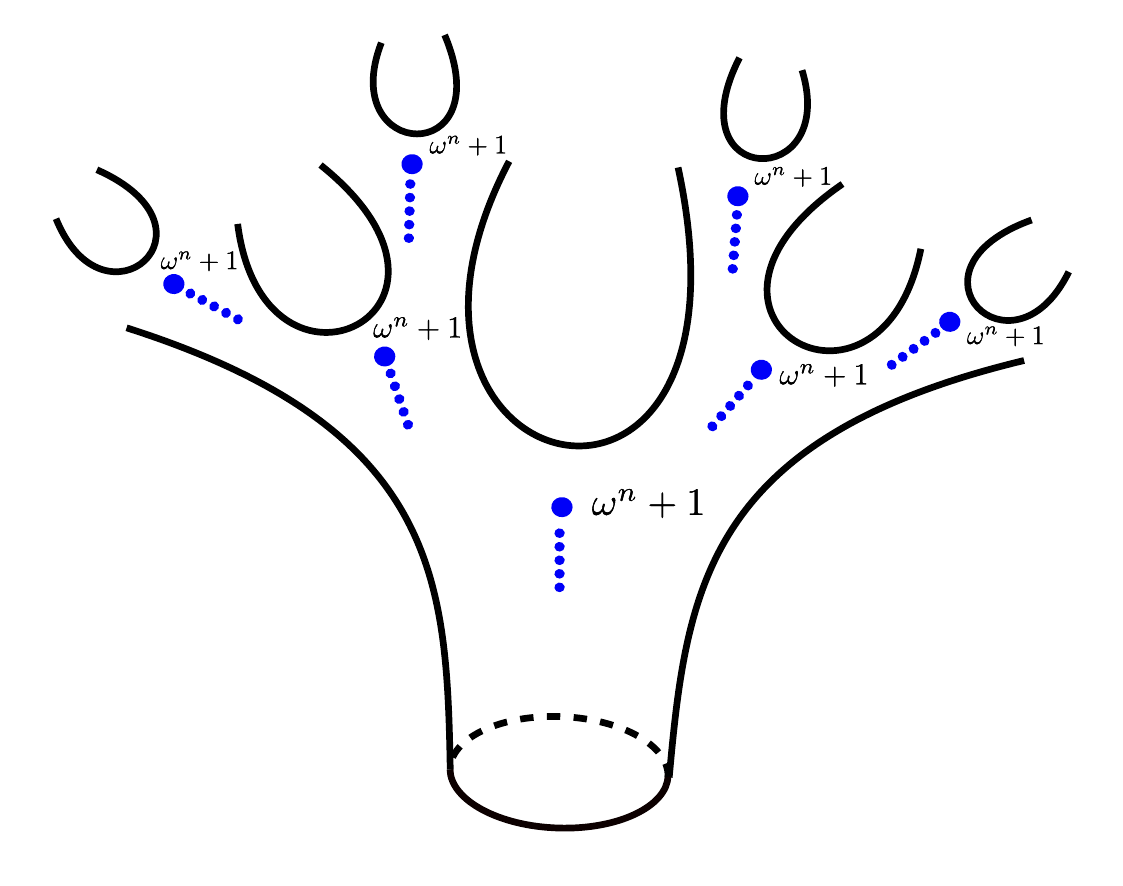}
	\caption{Surface $D_{n+1}$ with one boundary component and whose space of ends has Cantor-Bendixson rank $n+1$ with $(n+1)$th-derived set homeomorphic to a Cantor set $C_{n+1}$ and, each point of $C_{n+1}$ is accumulated by points homeomorphic to $\omega^{n}+1$.}
	\label{Fig:Pre-Ejemplo}
\end{center}
\end{figure}

Let $T$ be the surface obtained from  the Cantor tree surface by placing $2^n$ copies of $D_n$ at each $n$th-level of the tree, one for each bifurcation. Let $C$ denote the set of ends of $T$ coming from the ends of the Cantor tree surface. Observe that each end $x$ in $C$ does not have stable neighborhoods. Indeed, take $U$ a neighborhood of $x$ and let $n_U$ be the smallest natural number such that $C_{n_U}\cap \Ends(U) \neq \emptyset$. By construction, we can find a neighborhood $V\subseteq U$ of $x$ with $\Ends(V)$ not containing points of $C_{n_U}$. Given that the ends in $C_n$ are not comparable with the ends in $C_m$ if $n\neq m$, $V$ cannot contain homeomorphic copies of $U$. % Notice that $T$ has a set of maximal ends that is homeomorphic to a Cantor set such that each of such end has no stable neighborhoods. %Note that the set of maximal ends of $T$ is homeomorphic to a Cantor set and, importantly, each maximal end of $T$ has no stable neighbors.

Let $g$ be a nonzero natural number and $F_g$ be the surface of genus $g$ and space of ends homemorphic to $\omega + 1$. Finally, we define $\Sigma$ to be the surface obtained from $F_g$ by replacing a neighborhood of each of its isolated ends with a copy of the surface $T$; see Figure \ref{Fig:Ejemplo}. 

Note that the space of ends of $\Sigma$ is self-similar, $\Sigma$ has a unique maximal end and therefore the unique maximal end of $\Sigma$ has stable neighborhoods. Moreover, each maximal end of $T$ is an immediate predecessor end of the unique maximal end of $\Sigma$ and, more importantly, some of them have no stable neighborhoods  as we explained before. Therefore the surface $\Sigma$ is not tame. Since $\Sigma$ has finite nonzero genus, then $\Map(\Sigma)$ is not globally CB. Let $K$ be the compact subsurface of $\Sigma$ with two boundary components, of genus $g$ and complementary components $\Sigma_0$ and $\Sigma_1$, where $\Sigma_1$ is exactly the first copy $T$ and $\Sigma_0$ contains the remaining copies of $T$ in the construction of $\Sigma$. The reader can verify that $K$ is the desired subsurface in Theorem \ref{TEO:BigMapsLocallyCB1maxEnd} for which $\mathcal{V}_K$ defines a CB-neighborhood of the identity. So, $\Map(\Sigma)$ is locally CB and therefore it is CB generated by Theorem \ref{THM:CBGenerated}. \qed

\begin{figure}[!ht]
\begin{center}
	\includegraphics[width=.9\textwidth]{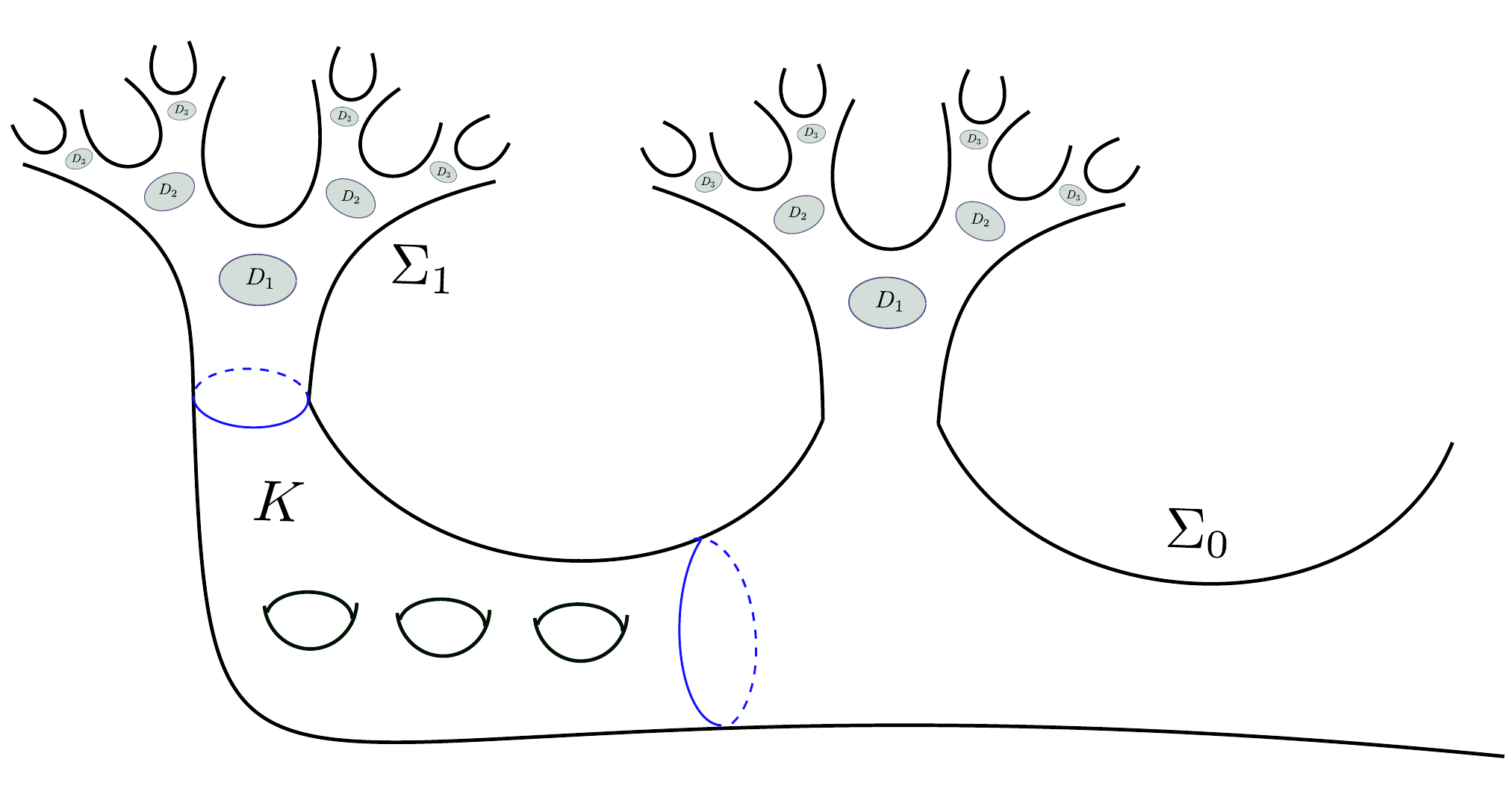}
	\caption{\small Non-tame surface with a unique maximal end and whose mapping class group is CB generated but not globally CB.}
	\label{Fig:Ejemplo}
\end{center}
\end{figure}

\bibliography{biblio}

\end{document}